\documentclass[12pt,reqno]{amsart}
\usepackage{amssymb,amscd,amsthm,latexsym}

\usepackage[all]{xy}
\usepackage{color}

\usepackage[all]{xy}
\usepackage{color}
\usepackage{graphicx}

\theoremstyle{plain}
\newtheorem{theorem}[subsection]{Theorem}
\newtheorem{lemma}[subsection]{Lemma}
\newtheorem{proposition}[subsection]{Proposition}
\newtheorem{corollary}[subsection]{Corollary}

\theoremstyle{definition}

\newtheorem{remark}[subsection]{Remark}

\newdir{ >}{% end of arrow for the monos
	@{}*!/-10pt/@{>} }

\newcommand{\C}{\ensuremath{\mathbb{C}}}
\newcommand{\V}{\ensuremath{\mathbb{V}}}

\newcommand{\Eq}{ \ensuremath{\mathrm{Eq}} }
\newcommand{\Cong}{ \ensuremath{\mathrm{Cong}} }
\newcommand{\Equiv}{ \ensuremath{\mathrm{Equiv}} }

\newcommand{\mono}[2]{ \ensuremath{ \xymatrix@1@C=15pt{ #1\; \ar@{ >->}[r] & #2 } } }
\newcommand{\regepi}[2]{ \ensuremath{ \xymatrix@1@C=15pt{ #1 \ar@{>>}[r] & #2 } } }

\def\pullback{% with thanks to Valerian Even
	\ar@{-}[]+R+<6pt,-1pt>;[]+RD+<6pt,-6pt>%
	\ar@{-}[]+D+<1pt,-6pt>;[]+RD+<6pt,-6pt>}

\usepackage{hyperref}
\begin{document}
\title[Facets of congruence distributivity in Goursat categories]%
{Facets of congruence distributivity in Goursat categories}
	
\author[Marino Gran]{Marino Gran}
\address{(\textbf{Marino Gran}), Institut de Recherche en Math\'ematique et Physique, Universit\'e Catholique de Louvain, Chemin du Cyclotron 2,
1348 Louvain-la-Neuve, Belgium}
\email{marino.gran@uclouvain.be}

\author[Diana Rodelo]{Diana Rodelo}
\address{(\textbf{Diana Rodelo}), Departamento de Matem\'atica, Faculdade de Ci\^{e}ncias e Tecnologia, Universidade
do Algarve, Campus de Gambelas, 8005-139 Faro, Portugal and CMUC, Department of Mathematics, University of Coimbra, 3001-501 Coimbra, Portugal}
\thanks{The second author acknowledges partial financial assistance by the Centre for Mathematics of the
University of Coimbra -- UID/MAT/00324/2019, funded by the Portuguese Government through
FCT/MEC and co-funded by the European Regional Development Fund through the Partnership
Agreement PT2020.}
\email{drodelo@ualg.pt}
	
\author[Idriss Tchoffo Nguefeu]{Idriss Tchoffo Nguefeu}
\address{(\textbf{Idriss Tchoffo Nguefeu}), Institut de Recherche en Math\'ematique et Physique, Universit\'e Catholique de Louvain, Chemin du Cyclotron 2,
1348 Louvain-la-Neuve, Belgium}
\email{idriss.tchoffo@uclouvain.be}


 % Exac

%%%%%%%%%%%%%%%%%%%%%%%%%%%%%%%%%%%%%%%  ABSTRACT  %%%%%%%%%%%%%%%%%%%%%%%%%%%%%%%%%%%%%%%%%%%%%%%%%%%%%%%%%%%%%%%%%%%%
\begin {abstract}We give new characterisations of regular Mal'tsev categories with distributive lattice of equivalence relations through variations of the so-called \emph{Triangular Lemma} and \emph{Trapezoid Lemma} in universal algebra. We then give new characterisations of equivalence distributive Goursat categories (which extend $3$-permutable varieties) through variations of the Triangular and Trapezoid Lemmas involving reflexive and positive relations.
\end {abstract}

\date{\today}

\maketitle

{\section*{Introduction}}

Regular \emph{Mal'tsev categories} \cite{CLP} extend $2$-permutable varieties of universal algebras, also including many examples which are not necessarily varietal, such as topological groups, compact groups, torsion-free groups and $\mathsf{C}^*$-algebras, for instance.
These categories have the property that any pair of (internal) equivalence relations $R$ and $S$ on the same object permute: $RS = SR$ (see \cite{BGJ}, for instance, and the references therein). It is well known that regular Mal'tsev categories have the property that the
lattice of equivalence relations on any object is modular, so that they satisfy (the categorical version of) Gumm's \emph{Shifting Lemma}~\cite{Gumm}.
More generally, this is the case for \emph{Goursat categories} \cite{CKP}, which are those regular categories for which the composition of equivalence relations on the same object is $3$-permutable: $RSR=SRS$.

In~\cite{GRT} we proved that, for a regular category, the property of being a Mal'tsev category, or of being a Goursat category, can be both characterised through suitable variations of the Shifting Lemma.
These variations considered the Shifting Lemma for relations which were not necessarily equivalence relations, but only reflexive or positive \cite{Sel} ones, thus giving rise to stronger versions of the Shifting Lemma: the main part of those characterisations was to show that these stronger versions implied $2$-permutability or $3$-permutabili\-ty.

There are other properties a variety may possess which can be expressed similarly, as for instance the \emph{distributivity} of the lattice of congruences. These properties are related to the Shifting Lemma, and are called the \emph{Triangular Lemma} and the \emph{Trapezoid Lemma} in the varietal context \cite{Trapezoid}. These properties were first introduced in~\cite{Duda1,Duda2} where the Trapezoid Lemma was called the \emph{Upright Principle}.
This led us to further study the connections between these results and the property, for a regular Mal'tsev (or a Goursat) category, of having distributive equivalence relation lattices on any of its objects.

From~\cite{Trapezoid} we know that, for a variety $\V$ of universal algebras, the fact that both the Shifting Lemma and the Triangular Lemma hold in $\V$ is equivalent to $\V$ being a congruence distributive variety, and is also equivalent to the fact that the Trapezoid Lemma holds in $\V$. Consequently, by considering stronger versions of the Triangular Lemma we were hoping to get at once $2$-permutability (or $3$-permutability) and congruence distributivity in a varietal context, and to extend these observations to a categorical context.

Explaining how this is indeed possible is the main goal of this paper, where suitable variations of the Triangular Lemma and of the Trapezoid Lemma are shown to be the right properties to characterise \emph{equivalence distributive} categories (the natural generalisation of congruence distributive varieties). More precisely, when $\C$ is a regular Mal'tsev category, or even a Goursat category, the Triangular Lemma is equivalent to the Trapezoid Lemma, and both of them are equivalent to $\C$ being equivalence distributive (Propositions~\ref{Mal'tsev and dist} and~\ref{Goursat and dist}). We also give new characterisations of equivalence distributive Mal'tsev categories
through variations of the Triangular Lemma and of the Trapezoid Lemma  (Theorem~\ref{Mal'tsev+dist <=> TTrL}), which then apply to arithmetical varieties~\cite{Pixley} and arithmetical categories~\cite{Pedic}. Inspired by the ternary Pixley term of arithmetical varieties~\cite{Pixley}, we consider a condition for relations, stronger than difunctionality~\cite{Riguet}, which captures the property for a regular category to be a Mal'tsev and equivalence distributive one (Theorem \ref{ternary difunctionality}). In the last section we characterise equivalence distributive Goursat categories (Theorem~\ref{Goursat <=> TTrL}) through variations on the Triangular and Trapezoid Lemmas involving reflexive and positive relations.

\section{Shifting Lemma, Triangular Lemma and Trapezoid Lemma}\label{lemmas}

For a variety $\V$ of universal algebras, Gumm's \emph{Shifting Lemma}~\cite{Gumm} is stated as follows. Given congruences $R,S$ and $T$ on the same algebra $X$ in $\V$ such that $R\wedge S\leqslant T$, whenever $x,y,u,v$ are elements in $X$ with $(x,y)\in R \wedge T$, $(x,u) \in S$, $(y,v)\in S$ and $(u,v)\in R$, it then follows that $(u,v) \in T$. We display this condition as
\begin{equation}\label{SL}
\vcenter{\xymatrix@C=30pt{
		x \ar@{-}[r]^-S \ar@{-}[d]^-R \ar@(l,l)@{-}[d]_-T & u \ar@{-}[d]_-R \ar@(r,r)@{--}[d]^-T \\
		y \ar@{-}[r]_-S  & v. }}
\end{equation}

A variety $\mathbb V$ of universal algebras satisfies the Shifting Lemma precisely when it is congruence modular~\cite{Gumm}, this meaning that the lattice of congruences $\Cong(X)$ on any algebra $X$ in $\mathbb V$ is modular.

A variety $\mathbb V$ of universal algebras satisfies the \emph{Triangular Lemma}~\cite{Trapezoid} if, given congruences $R,S$ and $T$ on the same algebra $X$ in $\mathbb V$ such that $R\wedge S\leqslant T$, whenever $y,u,v$ are elements in $X$ with $(u,y) \in T$, $(y,v)\in S$ and $(u,v)\in R$, it then follows that $(u,v) \in T$. We display this condition as
\begin{equation}\label{TL}
\vcenter{\xymatrix@C=30pt{
		& u \ar@{-}[d]^-R \ar@{-}[dl]_-T \ar@(r,r)@{--}[d]^-T \\
		y \ar@{-}[r]_-S  & v. }}
\end{equation}

A variety $\mathbb V$ of universal algebras satisfies the \emph{Trapezoid Lemma}~\cite{Trapezoid} if, given congruences $R,S$ and $T$ on the same algebra $X$ in $\mathbb V$ such that $R\wedge S\leqslant T$, whenever $x,y,u,v$ are elements in $X$ with $(x,y)\in T$, $(x,u) \in S$, $(y,v)\in S$ and $(u,v)\in R$, it then follows that $(u,v) \in T$. We display this condition as
\begin{equation}\label{TpL}
\vcenter{\xymatrix@C=30pt{
		& x \ar@{-}[dl]_-T \ar@{-}[r]^-S & u \ar@{-}[d]_-R \ar@(r,r)@{--}[d]^-T \\
		y \ar@{-}[rr]_-S  & & v. }}
\end{equation}

If the Trapezoid Lemma holds in a variety, then also the Shifting Lemma and the Triangular Lemma hold, since they are weaker.

A categorical version of the Shifting Lemma (stated differently from the original formulation recalled above) may be considered in any finitely complete category, and this leads to the notion of a \emph{Gumm category}~\cite{BG0, BG1}.
One can easily check that both the properties expressed by the Triangular Lemma and by the Trapezoid Lemma only involve finite limits. It is then possible to speak of the validity of these properties in any finitely complete category.
Nevertheless, since the main results of this paper will be obtained in $3$-permutable (=Goursat) categories \cite{CKP}, we shall need to be able to compose relations. For this reason we shall always require that the base category $\mathbb C$ is regular.\\

Recall that a finitely complete category $\mathbb C$ is \emph{regular} \cite{Barr} if any arrow $f \colon A \rightarrow B$ has a factorisation as a regular epimorphism (=a coequaliser) $p \colon A \rightarrow I$ followed by a monomorphism $m \colon I \rightarrow B$, and these factorisations are pullback stable.
The subobject determined by the monomorphism $m \colon I \rightarrow B$ is unique, and it is called the \emph{regular image} of the arrow $f$.

%In a regular context one can use set-theoretic terms thanks to Barr's embedding theorem~\cite{Barr} (see also Metatheorem A.5.7 in~\cite{BB}), so that the properties given in diagrams~\eqref{SL}, \eqref{TL} and \eqref{TpL} may still be expressed by using generalised elements.

In a regular category, it is possible to compose relations. If $(R, r_1, r_2)$ is a relation from $X$ to $Y$ and $(S, s_1, s_2)$ a relation from $Y$ to $Z$, their composite $S R$ is a relation from $X$ to $Z$ obtained as the regular image of the arrow $$( r_1 \pi_1, s_2 \pi_2 ) \colon R \times_Y S \rightarrow X \times Z,$$ where $(R \times_Y S, \pi_1, \pi_2)$ is the pullback of $r_2$ along $s_1$. The composition of relations is then associative, thanks to the fact that regular epimorphisms are assumed to be pullback stable.

In a regular category $\C$, given equivalence relations $R,S$ and $T$ on the same object $X$ such that $R\wedge S\leqslant T$, the lemmas recalled above can be interpreted as follows:\\

\vspace{5pt}
\begin{tabular}{rcl}
	\textbf{Shifting Lemma}:  & \hspace{30pt}  $R\wedge S(R\wedge T)S\leqslant T$ \hspace{35pt} & \textbf{(SL)} \vspace{5pt}\\
	\textbf{Triangular Lemma}: & \hspace{30pt}  $R\wedge ST\leqslant T$ \hspace{35pt} & \textbf{(TL)} \vspace{5pt}\\
	\textbf{Trapezoid Lemma}: & \hspace{30pt} $R\wedge STS\leqslant T$ \hspace{35pt} & \textbf{(TpL)}
\end{tabular}
\
\vspace{4mm}

We would like to point out that in some recent papers the notion of \emph{majority category}
has been introduced and investigated \cite{Hoefnagel, Hoef2}. This notion is closely related to the validity of the properties just recalled. For a regular category $\mathbb C$, the property of being a majority category can be equivalently defined as follows (see \cite{Hoef2}): for any reflexive relations $R$, $S$ and $T$ on the same object $X$ in $\mathbb C$, the inequality
$$R \wedge (ST ) \leqslant (R\wedge S) (R\wedge T) $$
holds. We then observe that any regular majority category satisfies the Trapezoid Lemma (and, consequently, also the weaker Triangular Lemma and Shifting Lemma):
\begin{lemma}\label{TrapMaj}\cite{Hoef3}
	The Trapezoid Lemma holds true in any regular majority category $\mathbb C$ .
\end{lemma}
\begin{proof}
	Given equivalence relations $R$, $S$ and $T$ on the same object such that $R \wedge S \leqslant T$, then $$R \wedge (STS) \leqslant (R\wedge S) (R \wedge (TS)) \leqslant T (R\wedge T) (R \wedge S) \leqslant TTT = T.$$
\end{proof}

\section{$2$-permutability and $3$-permutability}\label{perm}
A variety $\V$ of universal algebras is \emph{$2$-permutable}~\cite{Smith} when, given any congruences $R$ and $S$ on the same algebra $X$, we have the equality $RS=SR$. Such varieties are characterised by the existence of a ternary operation $p$ such that $p(x,y,y)=x=p(y,y,x)$~\cite{Malcev}. A variety $\V$ of universal algebras is called \emph{$3$-permutable} when the strictly weaker equality $RSR=SRS$ holds. Such varieties are characterised by the existence of two quaternary operations $p$ and $q$ satisfying the identities
\begin{equation}\label{quaternary}
\begin{array}{l}
p(x,y,y,z)=x \\
p(u,u,v,v)=q(u,u,v,v) \\
q(x,y,y,z)=z
\end{array}
\end{equation}
(see~\cite{HM}).

The notions of $2$-permutability and $3$-permutability can be extended from varieties to regular categories by replacing congruences with (internal) equivalence relations, allowing one to explore some interesting new (non-varietal) examples. Regular categories that are $2$-permutable and $3$-permutable are usually called \emph{Mal'tsev categories}~\cite{CLP} and \emph{Goursat categories} \cite{CKP}, respectively. As examples of regular Mal'tsev categories that are not (finitary) varieties of algebras we list: $\mathrm{C}^*$-algebras, compact groups, topological groups~\cite{CKP}, torsion-free abelian groups, reduced commutative rings, cocommutative Hopf algebras over a field~\cite{GSV}, any abelian category, and the dual of any topos~\cite{CKP}. Any regular Mal'tsev category is a Goursat category. As examples of Goursat categories that are not regular Mal'tsev categories we have the category of implication algebras \cite{Mit} and the category of right complemented semigroups \cite{HM}.

It is well-known that any $2$-permutable or $3$-permutable variety is congruence modular~\cite{Gumm, Jonsson53}, thus the Shifting Lemma holds. This result also extends to the regular categorical context. First note that in a regular category $\C$, the preordered set $\Equiv(X)$ of equivalence relations on an object $X$ in $\C$ is just a meet semilattice. If $\C$ is a Mal'tsev or a Goursat category, then the existence of binary joins is guaranteed (see Theorems~\ref{Mal'tsev chars}(iii) and~\ref{Goursat chars}(iii)), so that $\Equiv(X)$ is a lattice which is, moreover, modular~\cite{CKP}. The modularity of the lattices of equivalence relations implies that the Shifting Lemma holds in $\C$. However, the converse fails to be true even in the case of a variety of infinitary algebras, as it was shown in Example 12.5 in~\cite{GJ}.

Regular Mal'tsev and Goursat categories are also characterised by other properties on (equivalence) relations, as follows:

\begin{theorem}\label{Mal'tsev chars} \emph{\cite{CLP}}
	Let $\C$ be a regular category. The following conditions are equivalent:
	\begin{enumerate}
		\item[(i)] $\C$ is a Mal'tsev category;
		\item[(ii)] $\forall R,S\in \Equiv(X)$, $RS\in \Equiv(X)$, for any object $X$ in $\C$;
		\item[(iii)] $\forall R,S\in \Equiv(X)$, $R\vee S=RS (=SR)$, for any object $X$ in $\C$;
		\item[(iv)] any reflexive relation $E$ is symmetric: $E^{\circ}=E$;
		\item[(v)] any relation $D$ is difunctional: $DD^{\circ}D=D$.
		
	\end{enumerate}
\end{theorem}

\begin{theorem}\label{Goursat chars} \emph{\cite{CKP}}
	Let $\C$ be a regular category. The following conditions are equivalent:
	\begin{enumerate}
		\item[(i)] $\C$ is a Goursat category;
		\item[(ii)] $\forall R,S\in \Equiv(X)$, $RSR\in \Equiv(X)$, for any object $X$ in $\C$;
		\item[(iii)] $\forall R,S\in \Equiv(X)$, $R\vee S=RSR (=SRS)$, for any object $X$ in $\C$;
		\item[(iv)] any relation $P$ is such that $PP^{\circ}PP^{\circ}=PP^{\circ}$;
		\item[(v)] any reflexive relation $E$ is such that $EE^{\circ}=E^{\circ}E$.
	\end{enumerate}
\end{theorem}

\section{Equivalence distributivity}\label{dist}
A lattice $L$ is called \emph{distributive} when
$$
a\wedge (b\vee c) = (a\wedge b) \vee (a\wedge c), \forall a,b,c \in L.
$$
Equivalently, $L$ is distributive if and only if it satisfies the Horn sentence
\begin{equation}\label{equivalent to dist}
a\wedge b \leqslant c \Rightarrow a\wedge (b\vee c)\leqslant c.
\end{equation}
A variety $\V$ of universal algebras is called \emph{congruence distributive} when the lattice $\Cong(X)$ of congruences on any algebra $X$ in $\V$ is distributive. Similarly, we shall call a regular category $\C$ \emph{equivalence distributive} when the meet semilattice $\Equiv(X)$ of equivalence relations is a distributive lattice, for all objects $X$ in $\C$.

Any distributive variety gives an example of an equivalence distributive category. The varieties of Boolean algebras, Heyting algebras and Von Neumann regular rings \cite{GR}, or the dual of any (pre)topos are also examples. These are actually \emph{arithmetical categories}~\cite{Pedic}, i.e. Barr-exact Mal'tsev equivalence distributive categories. Recall that a Barr-exact category $\mathbb C$ is a regular category where any equivalence relation in $\mathbb C$ is effective, i.e. the kernel pair of some arrow \cite{Barr}.
%{\color{red} More ???}

The congruence distributive varieties can be characterised as follows:

\begin{theorem}~\cite{Trapezoid}\label{Sl+TL=TpL} Let $\V$ be a variety of universal algebras. The following conditions are equivalent:
	\begin{itemize}
		\item[(i)] $\V$ is congruence distributive;
		\item[(ii)] the Trapezoid Lemma holds in $\V$;
		\item[(iii)] the Shifting Lemma and the Triangular Lemma hold in $\V$.
	\end{itemize}
\end{theorem}

The equivalence between the Triangular Lemma and Trapezoid Lemma holds for any algebra $X$ which is \emph{congruence permutable}, meaning that $2$-permutability holds in $\Cong(X)$:

\begin{proposition}~\cite{Trapezoid}\label{TL<=>TpL, for Cong(X)} Let $\V$ be a variety of universal algebras and $X$ a congruence permutable algebra. The following conditions are equivalent:
	\begin{itemize}
		\item[(i)] the Triangular Lemma holds for $X$;
		\item[(ii)] the Trapezoid Lemma holds for $X$;
		\item[(iii)] $\Cong(X)$ is distributive.
	\end{itemize}
\end{proposition}

This result may be extended to the context of regular categories. To do so we apply Barr's Theorem~\cite{Barr} which allows us to use part of the internal logic of a topos to develop proofs in a regular category. In particular, finite limits can be described elementwise as in the category of sets and regular epimorphisms via the usual formula describing surjections (see also Metatheorem A.5.7 in~\cite{BB}).

\begin{proposition}\label{Mal'tsev and dist}
	Let $\C$ be a regular Mal'tsev category. The following conditions are equivalent:
	\begin{enumerate}
		\item[(i)] the Triangular Lemma holds in $\C$;
		\item[(ii)] the Trapezoid Lemma holds in $\C$;
		\item[(iii)] $\C$ is equivalence distributive.
	\end{enumerate}
\end{proposition}
\begin{proof}
	(i) $\Rightarrow$ (ii) Let $R,S$ and $T$ be equivalence relations on an object $X$ such that $R\wedge S\leqslant T$ and suppose that $x,y,u,v$ are related as in \eqref{TpL}. Since $\C$ is a Mal'tsev category, then $TS$ is an equivalence on $X$ (Theorem~\ref{Mal'tsev chars}(ii)). We may apply the Triangular Lemma to
	$$
	
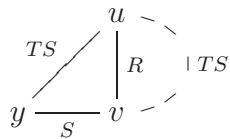
\xymatrix{ & u \ar@{-}[dl]_-{TS} \ar@{-}[d]^-R \ar@(r,r)@{--}[d]^-{TS} \\
		y \ar@{-}[r]_-S & v}
	$$
	($R\wedge S\leqslant T\leqslant TS$), to conclude that $(u,v)\in TS (=ST)$. So, there exists $a$ in $X$ such that
	$$
	\xymatrix{ & u \ar@{-}[dl]_-T \ar@{-}[d]^-R \ar@(r,r)@{--}[d]^-T \\
		a \ar@{-}[r]_-S & v.}
	$$
	Applying the Triangular Lemma again, we conclude that $(u,v)\in T$.
	
	(ii) $\Rightarrow$ (iii) We prove that \eqref{equivalent to dist} holds with respect to the lattice $\Equiv(X)$ of equivalence relations on an object $X$. Let $R,S,T\in \Equiv(X)$ be such that $R\wedge S\leqslant T$. Then
	$$
	\begin{array}{lcll}
	R\wedge (S\vee T) & = & R\wedge ST, & \mathrm{by\;\;} \mathrm{Theorem\;\;}\ref{Mal'tsev chars}\mathrm{(iii)} \\
	& \leqslant & R\wedge STS \\
	& \leqslant & T, & \mathrm{by\;\;} \mathrm{\textbf{(TpL)}}.
	\end{array}
	$$
	(iii) $\Rightarrow$ (ii) Let $R,S$ and $T$ be equivalence relations in $\Equiv(X)$ such that $R\wedge S\leqslant T$. Then
	$$
	\begin{array}{lcll}
	R\wedge STS & \leqslant & R\wedge (S\vee T)\\
	& \leqslant & T, & \mathrm{by\;\;}\eqref{equivalent to dist} \\
	\end{array}
	$$
	thus \textbf{(TpL)} holds.
	
	(ii) $\Rightarrow$ (i) Obvious.
\end{proof}

Note that the implications (iii) $\Rightarrow$ (ii) $\Rightarrow$ (i) of Proposition~\ref{Mal'tsev and dist} hold in any regular category.

\begin{remark}
	It is known from Corollary 3.2 in~\cite{{Hoef2}} that a regular Mal'tsev category $\C$ is equivalence distributive if and only if $\C$ is  a majority category. That every Mal'tsev equivalence distributive category is a majority category was already known from~\cite{Hoefnagel}. We remark that the converse implication also easily follows from Lemma~\ref{TrapMaj} and Proposition~\ref{Mal'tsev and dist}.
\end{remark}

Next we show that the same equivalent conditions hold in the weaker context of Goursat categories. The most difficult implication to prove is that a Goursat category which satisfies the Triangular Lemma also satisfies the Trapezoid Lemma. We start by giving a direct proof of this fact in the varietal context to then obtain a categorical translation of the proof via matrix conditions~\cite{JRVdL}. Note that, for varieties, this result actually follows from Theorem 1 in~\cite{Trapezoid}; however, we give an alternative proof which is suitable to be extended to the categorical context of regular categories.

\begin{lemma}\label{3-perm + TL => TpL} If $\V$ is a $3$-permutable variety which satisfies the Triangular Lemma, then the Trapezoid Lemma also holds in $\V$.
\end{lemma}
\begin{proof} Let $R, S$ and $T$ be congruences on the same algebra $X$ in $\V$ such that $R\wedge S\leqslant T$. Suppose that $x,y,u,v$ are elements in $X$ related as in \eqref{TpL}. From the relations
	\begin{equation}\label{A}
	\begin{array}{l}
	x T x S x R x \\
	x T x S u R u \\
	x T y S v R u  \\
	y T y S y R y,
	\end{array}
	\end{equation}
	we may deduce the following ones by applying the quaternary operations $p$ and $q$ (see~\eqref{quaternary}), respectively:
	$$
	x T p(x,x,y,y) S p(x,u,v,y) R x
	$$
	and
	$$
	y T q(x,x,y,y) S q(x,u,v,y) R y.
	$$
	We apply the Triangular Lemma to
	\begin{equation}\label{t1}
	\vcenter{\xymatrix@R=40pt{ & x \ar@{-}[dl]_-T \ar@{-}[d]^-R \ar@(r,r)@{--}[d]^(.5)T \\
			p(x,x,y,y) \ar@{-}[r]_-S & p(x,u,v,y)}}
	\end{equation}
	and
	\begin{equation}\label{t2}
	\vcenter{\xymatrix@R=40pt{ & y \ar@{-}[dl]_-T \ar@{-}[d]^-R \ar@(r,r)@{--}[d]^(.5)T \\
			q(x,x,y,y) \ar@{-}[r]_-S & q(x,u,v,y).}}
	\end{equation}
	
	Next, we apply the Shifting Lemma to
	\begin{equation}\label{SL1}
	\vcenter{\xymatrix@R=40pt{ x \ar@{-}[d]^-R \ar@(l,l)@{-}[d]_(.5)T \ar@{}[d]_-{\eqref{t1}} \ar@{-}[r]^-S & u =p(u,u,u,v) \ar@{-}[d]_-R \ar@(r,r)@{--}[d]^-T \\
			p(x,u,v,y) \ar@{-}[r]_-S & p(u,u,v,v)}}
	\end{equation}
	and
	\begin{equation}\label{SL2}
	\vcenter{\xymatrix@R=40pt{ y \ar@{-}[d]^-R \ar@(l,l)@{-}[d]_(.5)T \ar@{}[d]_-{\eqref{t2}} \ar@{-}[r]^-S & v =q(u,u,u,v)\ar@{-}[d]_-R \ar@(r,r)@{--}[d]^-T \\
			q(x,u,v,y) \ar@{-}[r]_-S & q(u,u,v,v).}}
	\end{equation}
	From \eqref{SL1} and ~\eqref{SL2}, we obtain $u T p(u,u,v,v)=q(u,u,v,v) T v$; it follows that $(u,v)\in T$.
\end{proof}

We adapt this varietal proof into a categorical one using an appropriate matrix and the corresponding relations which may be deduced from it (see~\cite{JRVdL} for more details). The kind of matrix we use translates the quaternary identities~\eqref{quaternary} into the property on relations given in Theorem~\ref{Goursat chars}(iv):
\begin{equation}\label{matrix}
\left( \begin{array}{cccc|cc}
x & y & y & z & x & z \\
u & u & v & v & \alpha & \alpha
\end{array}\right)
\end{equation}
The first and second columns after the vertical separation in the matrix are the result of applying $p$ and $q$, respectively, to the elements in the lines before the vertical separation. Thus, the introduction of a new element $\alpha$, to represent the identity $p(u,u,v,v)=q(u,u,v,v) (=\alpha)$. We then ``interpret'' the matrix as giving relations between top elements and bottom elements as follows. Whenever the relations before the vertical separation in the matrix are assumed to hold, then we may conclude that the relations after the vertical separation also hold. For this matrix, the interpretation gives: for any binary relation $P$, if $xPu, yPu, y P v$ and $zPv$, then $x P \alpha$ and $z P \alpha$, for some $\alpha$; this gives the property $PP^{\circ}PP^{\circ}\leqslant PP^{\circ}$. Since $PP^{\circ}\leqslant PP^{\circ}PP^{\circ}$ is always true, we get precisely $PP^{\circ}PP^{\circ}= PP^{\circ}$ from  Theorem~\ref{Goursat chars}(iv).

\begin{proposition}\label{Goursat and dist}
	Let $\C$ be a Goursat category. The following conditions are equivalent:
	\begin{enumerate}
		\item[(i)] the Triangular Lemma holds in $\C$;
		\item[(ii)] the Trapezoid Lemma holds in $\C$;
		\item[(iii)] $\C$ is equivalence distributive.
	\end{enumerate}
\end{proposition}
\begin{proof}
	(i) $\Rightarrow$ (ii) We extend the proof of Lemma~\ref{3-perm + TL => TpL} to a categorical context by constructing an appropriate matrix of the type~\eqref{matrix}. In that proof we applied $p$ and $q$ to the $4$-tuples $(x,x,x,y)$, $(x,x,y,y)$, $(x,u,u,y)$, $(u,u,u,v)$ and $(u,u,v,v)$. We put them in the matrix so that $(x,x,x,y)$, $(x,u,u,y)$ and $(u,u,u,v)$ go to the top lines and $(x,x,y,y)$ and $(u,u,v,v)$ go to the bottom lines as follows
	$$
	\left(\begin{array}{cccc|cc}
	x & x & x & y & x & y \\
	x & u & u & y & x & y \\
	u & u & u & v & u & v \vspace{10pt}\\
	x & x & y & y & \alpha & \alpha \\
	u & u & v & v & \varepsilon & \varepsilon
	\end{array}\right)
	$$
	We also used the $4$-tuple $(x,u,v,y)$, but it does not ``fit'' into this type of matrix; it will be used in the definition of the binary relation $P$. From the matrix, we see that the relation $P$ should be defined from $X^3$ to $X^2$. The relations between the $4$-tuples in the matrix above and $(x,u,v,y)$ given in \eqref{A}, and the bottom and right hand relations in \eqref{SL1} and \eqref{SL2} tell us that $P$ should be defined as:
	$$
	(a,b,c) P (d,e) \Leftrightarrow \exists z\;\;\mathrm{such\;\;that}\;\; a T d S z R b, z S e\;\;\mathrm{and}\;\; e R c.
	$$
	From the matrix we see that $(x,x,u) PP^{\circ}PP^{\circ} (y,y,v)$, from which we conclude that $(x,x,u)PP^{\circ} (y,y,v)$. It then follows that $(x,x,u)P(\alpha,\varepsilon)$ and $(y,y,v)P(\alpha, \varepsilon)$, for some $(\alpha, \varepsilon)$, i.e. there exist $\beta$ and $\delta$ such that
	$$
	\begin{array}{l}
	x T \alpha S \beta R x,  \beta S \varepsilon \;\;\mathrm{and}\;\; \varepsilon R u \\
	y T \alpha S \delta R y, \delta S \varepsilon\;\;\mathrm{and}\;\; \varepsilon R v.
	\end{array}
	$$
	Next we apply the Triangular Lemma to
	\begin{equation}\label{t1'}
	\vcenter{\xymatrix{ & x \ar@{-}[dl]_-T \ar@{-}[d]^-R \ar@(r,r)@{--}[d]^-T \\
			\alpha \ar@{-}[r]_-S & \beta}}
	\end{equation}
	and
	\begin{equation}\label{t2'}
	\vcenter{\xymatrix{ & y \ar@{-}[dl]_-T \ar@{-}[d]^-R \ar@(r,r)@{--}[d]^-T \\
			\alpha \ar@{-}[r]_-S & \delta.}}
	\end{equation}
	We now apply the Shifting Lemma to
	\begin{equation}\label{SL1'}
	\vcenter{\xymatrix@C=30pt{ x \ar@(l,l)@{-}[d]_-T \ar@{-}[r]^-S \ar@{-}[d]^-R \ar@{}[d]_-{\eqref{t1'}} & u \ar@{-}[d]_-R \ar@(r,r)@{--}[d]^-T \\
			\beta \ar@{-}[r]_-S & \varepsilon}}
	\end{equation}
	and
	\begin{equation}\label{SL2'}
	\vcenter{\xymatrix@C=30pt{ y \ar@(l,l)@{-}[d]_-T \ar@{-}[r]^-S \ar@{-}[d]^-R \ar@{}[d]_-{\eqref{t2'}} & v \ar@{-}[d]_-R \ar@(r,r)@{--}[d]^-T \\
			\delta \ar@{-}[r]_-S & \varepsilon.}}
	\end{equation}
	From \eqref{SL1'} and \eqref{SL2'} we obtain $uT \varepsilon T v$, thus $(u,v) \in T$.
	
	(ii) $\Rightarrow$ (iii)  We prove that \eqref{equivalent to dist} holds with respect to the lattice $\Equiv(X)$ of equivalence relations on an object $X$. Let $R,S,T\in \Equiv(X)$ be such that $R\wedge S\leqslant T$. Then
	$$
	\begin{array}{lcll}
	R\wedge (S\vee T) & = & R\wedge STS, & \mathrm{by\;\;} \mathrm{Theorem\;\;}\ref{Goursat chars}\mathrm{(iii)} \\
	& \leqslant & T, & \mathrm{by\;\;} \mathrm{\textbf{(TpL)}.}
	\end{array}
	$$
	
	The converse implications always hold in a regular context, as observed after the proof of Proposition~\ref{Mal'tsev and dist}.
\end{proof}

\begin{remark}
	In a varietal context, we know that the validity of the Shifting Lemma and the Triangular Lemma is equivalent to the validity of the Trapezoid Lemma (Theorem~\ref{Sl+TL=TpL}). We do not know if this result can be generalised to the context of a regular Gumm category \cite{BG0, BG1}. However,
	Propositions~\ref{Mal'tsev and dist} and \ref{Goursat and dist} show that this equivalence between the validity of the Triangular Lemma and the Trapezoid Lemma does hold under the stronger conditions that the base category is regular Mal'tsev and Goursat, respectively.
\end{remark}

\begin{remark}
	Note that another characterisation of regular Goursat categories which are equivalence distributive is given in \cite{DBourn}. A regular Goursat category is equivalence distributive if and only if the regular image of equivalence relations preserves binary meets: $f(R\wedge S)=f(R)\wedge f(S)$, for any regular epimorphism $f\colon X\twoheadrightarrow Y$ and $R,S\in \Equiv(X)$.
\end{remark}

\section{Equivalence distributive Mal'tsev categories}\label{edM}
In~\cite{GRT} we proved that regular Mal'tsev categories may be characterised through variations of the Shifting Lemma. Thanks to the results in the previous section we can now give some new characterisations of equivalence distributive Mal'tsev categories through similar variations of the Triangular and of the Trapezoid Lemmas.

The variations of the Triangular and of the Trapezoid Lemmas that we have in mind take $R,S$ or $T$ to be just reflexive relations. Note that, for diagrams such as \eqref{SL}, \eqref{TL} or \eqref{TpL}, where $R,S$ or $T$ are not symmetric, the relations are always to be considered from left to right and from top to bottom. To avoid ambiguity with the interpretation of such diagrams, from now on we will write $\xymatrix{ x \ar[r]^U &y} $
to mean that $(x,y) \in U$, whenever $U$ is a non-symmetric relation.

\begin{theorem}	
	\label{Mal'tsev+dist <=> TTrL} Let $\C$ be a regular category. The following conditions are equivalent:
	\begin{enumerate}
		\item[(i)] $\C$ is an equivalence distributive Mal'tsev category;
		\item[(ii)] the Trapezoid Lemma holds in $\C$ when $R,S$ and $T$ are reflexive relations;
		\item[(iii)] the Triangular Lemma holds in $\C$ when $R,S$ and $T$ are reflexive relations.
	\end{enumerate}
\end{theorem}

\begin{proof}
	(i) $\Rightarrow$ (ii)	Since $\C$ is a Mal'tsev category, reflexive relations are necessarily equivalence relations. Since $\C$ is also equivalence distributive, by Proposition \ref{Mal'tsev and dist}, the Trapezoid Lemma holds for any reflexive relations in $\C$.\\
	(ii) $\Rightarrow$ (iii)	is obvious.  \\
	(iii) $\Rightarrow$ (i) We follow the proof of Theorem 3.2 of~\cite{GRT} with respect to the implication: if the Shifting Lemma holds in $\C$ for reflexive relations, then $\C$ is a Mal'tsev category. The main issue is to fit the rectangle to which we applied the Shifting Lemma in that result, into a suitable triangle to which we shall now apply the Triangular Lemma (to get the same conclusion that $\C$ is a Mal'tsev category).
	
	To prove that $\C$ is a Mal'tsev category, we show that any reflexive relation $\langle e_1,e_2\rangle \colon E\rightarrowtail X\times X$ in $\C$ is also symmetric (Theorem~\ref{Mal'tsev chars}(iv)). Suppose that $(x,y)\in E$, and consider the reflexive relations $T$ and $R$ on $E$ defined as follows: \begin{center}$(aEb,cEd)\in R$ if and only if $(a,d)\in E$, and \\$(aEb, cEd)\in T$  if and only if $(c,b)\in E$.\end{center}
	The third reflexive relation on $E$ we consider is the \emph{kernel pair} $\Eq(e_2)$ of the second projection $e_2$.
	$\Eq(e_2)$ is an equivalence relation,  with the property that $\Eq(e_2)\leqslant R$ and $\Eq(e_2)\leqslant T$, so that $R\wedge \Eq(e_2)=\Eq(e_2)\leqslant T$. We can apply the assumption to the following relations given in solid lines
	$$
	\xymatrix@C=50pt{
		{} & xEx  \ar[dl]_-T  \ar[d]_-R \ar@(r,r)@{-->}[d]^-T \\
		xEy \ar[r]_-{\Eq(e_2)}  & yEy }
	$$
	($xEx$ and $yEy$ by the reflexivity of the relation $E$). We conclude that $(xEx,yEy)\in T$ and, consequently, that $(y,x)\in E$, so that $\C$ is a Mal'tsev category.
	
	Since the Triangular Lemma holds in $\C$, by Proposition \ref{Mal'tsev and dist} the category $\C$ is equivalence distributive.\\
	
\end{proof}
In the proof of the implication (iii) $\Rightarrow$ (i) we only used two ``genuine'' reflexive relations $R$ and $T$. This observation gives:
\begin{corollary} \label{Mal'tsev+dist <=> TTrL 2}
	Let $\C$ be a regular category. The following conditions are equivalent:	\begin{enumerate}
		\item[(i)] $\C$ is an equivalence distributive  Mal'tsev category;
		\item[(ii)] the Trapezoid Lemma holds in $\C$ when $R$ and $T$ are reflexive relations and $S$ is an equivalence relation;
		\item[(iii)] the Triangular Lemma holds in $\C$ when $R$ and $T$ are reflexive relations and $S$ is an equivalence relation.
	\end{enumerate}
\end{corollary}

\begin{remark}\label{arithmetical categories}
	An arithmetical category $\mathbb C$ is an equivalence distributive and Mal'tsev category which is, moreover, Barr-exact. Note that in this article we do not assume the existence of coequalisers, differently from what was done in
	Pedicchio's original definition of arithmetical category \cite{Pedic}. So, given a Barr-exact category $\C$, the same equivalent conditions stated in Theorem~\ref{Mal'tsev+dist <=> TTrL}(ii), Theorem~\ref{Mal'tsev+dist <=> TTrL}(iii), Corollary~\ref{Mal'tsev+dist <=> TTrL 2}(ii) and Corollary~\ref{Mal'tsev+dist <=> TTrL 2}(iii) give characterisations of the fact that $\C$ is an arithmetical category.
\end{remark}

We finish this section with a characterisation of equivalence distributive Mal'tsev categories through a property on ternary relations which is stronger than difunctionality (Theorem~\ref{Mal'tsev chars}(v)). The difunctionality of a binary relation $D\rightarrowtail X\times U$, $DD^{\circ}D=D$ can be pictured as
$$
\begin{array}{c}
x D u \\
y D u \\
y D v \\
\hline
x D v.
\end{array}
$$
Whenever the first three relations hold, we can conclude that the bottom relation $xDv$ holds.

Recall from~\cite{Pixley} that an arithmetical variety is such that there exists a Pixley term $p(x,y,z)$ such that $p(x,y,y)=x, p(x,x,y)=y$ and $p(x,y,x)=x$. We translate these Mal'tsev conditions into a property on relations (following the technique in~\cite{JanelidzeI}) which is expressed for ternary relations $D\rightarrowtail (X\times A)\times U$, seen as binary relations from $X\times A$ to $U$. It may be pictured as

\begin{equation}\label{DD}
\begin{array}{c}
(x,a) D u \\
(y,b) D u \\
(y,a) D v \\
\hline
(x,a) D v.
\end{array}
\end{equation}
This condition on the relation $D$ follows from applying the Pixley term to each column of elements, and writing the result in the bottom line.
In a regular context, property \eqref{DD} is equal to
$$
(\Eq(\pi_A)\wedge DD^{\circ}\Eq(\pi_X))D\leqslant D.
$$

\begin{theorem}\label{ternary difunctionality}
	Let $\C$ be a regular category. The following conditions are equivalent:
	\begin{enumerate}
		\item[(i)] $\C$ is an equivalence distributive  Mal'tsev category;
		\item[(ii)] any relation $D\rightarrowtail (X\times A)\times U$ has property \eqref{DD}.
	\end{enumerate}
\end{theorem}
\begin{proof}
	(i) $\Rightarrow$ (ii) Suppose that the first three relations in \eqref{DD} hold. Consider the equivalence relations $\Eq(d_1), \Eq(d_2)$ and $\Eq(d_3)$ on $D$ given by the kernel pairs of the projections of $D$. We have
	$$
	\begin{array}{rcl}
	(x,a,u) \,\Eq(d_2)\, (y,a,v) & \Rightarrow & (x,a,u)\, \Eq(d_1)\Eq(d_2)\, (y,a,v). \\
	(x,a,u)\, \Eq(d_3)\, (y,b,u)\, \Eq(d_1)\, (y,a,v) & \Rightarrow & (x,a,u) \,\Eq(d_1)\Eq(d_3)\, (y,a,v)
	\end{array}
	$$
	By assumption, $\Eq(d_1) (\Eq(d_2) \wedge \Eq(d_3) ) = (\Eq(d_1)\Eq(d_2))\wedge (\Eq(d_1)\Eq(d_3))$ (distributivity and Theorem~\ref{Mal'tsev chars}(iii)). Thus
	$$
	(x,a,u)\, \Eq(d_1)(\Eq(d_2)\wedge \Eq(d_3))\, (y,a,v),
	$$
	i.e.
	$$
	(x,a,u) \,\Eq(d_2) \wedge \Eq(d_3)\, (y,a,u)\, \Eq(d_1)\, (y,a,v)
	$$
	and, consequently, $(y,a,u)\in D$. Now we use the difunctionality of $D$ (Theorem~\ref{Mal'tsev chars}(v))
	$$
	\begin{array}{c}
	(x,a) D u \\
	(y,a) D u \\
	(y,a) D v \\
	\hline
	(x,a) D v,
	\end{array}
	$$
	to conclude that $(x,a) D v$.
	
	(ii) $\Rightarrow$ (i) The assumption applied to the case when $A=1$, is precisely difunctionality of any binary relation, so $\C$ is a Mal'tsev category (Theorem~\ref{Mal'tsev chars}(v)).
	
	Since $\C$ is a Mal'tsev category, we just need to prove the Triangular Lemma to conclude that $\C$ is equivalence distributive (Proposition~\ref{Mal'tsev and dist}). Consider equivalence relations $R,S$ and $T$ on an object $X$, such that $R\wedge S\leqslant T$ and that the relations in \eqref{TL} hold.
	
	We consider a relation $D\rightarrowtail (X\times X)\times X$ defined by
	$$
	(a,b) D c \Leftrightarrow \exists d\in X : dSa, dTb \;\;\mathrm{and}\;\; dRc.
	$$
	We have the following first three relations for $d=u, d=v$ and $d=y$, respectively,
	$$
	\begin{array}{c}
	(u,y) D v \\
	(y,v) D v \\
	(y,y) D y \\
	\hline
	(u,y) D y;
	\end{array}
	$$
	by assumption, we conclude that $(u,y)Dy$. By the definition of $D$, there exists $w\in X$ such that $wSu$, $wTy$ and $wRy$. We can then apply the Shifting Lemma to
	$$
	\vcenter{\xymatrix@C=30pt{
			w \ar@{-}[r]^-S \ar@{-}[d]^-R \ar@(l,l)@{-}[d]_-T & u \ar@{-}[d]_-R \ar@(r,r)@{--}[d]^-T \\
			y \ar@{-}[r]_-S  & v, }}
	$$
	to conclude that $uTv$.
\end{proof}

\section{Equivalence distributive Goursat categories}\label{edG}
In~\cite{GRT} we showed that Goursat categories may be characterised through variations of the Shifting Lemma. Together with the results from Section~\ref{edM}, we are going to characterise equivalence distributive Goursat categories through similar variations of the Triangular and the Trapezoid Lemmas. Such variations use the notion of positive relation.

\label{Positive relation} A relation $E$ on $X$  is called \textbf{positive} \cite{Sel} when it is of the form $ E = R^oR $ for some relation $R \rightarrowtail X \times Y$.

\begin{lemma}\label{lem}
	Let  $\mathbb{C}$ be a regular category. Then:
	\begin{enumerate}
		\item[(i)] any positive relation  is symmetric;
		\item[(ii)] any equivalence relation is positive.
		%	\item[(iii)]positive relations are stable under regular images.
	\end{enumerate}
	\begin{proof}
		$(i)$ Let $E$ be a positive relation and $R$ a relation such that $E = R^o R$.
		One has $E^o = (R^oR)^o = R^oR  = E.$
		
		$(ii)$ When $R$ is an equivalence relation, one has $R= R^oR.$
		
		%	$(iii)$ Let $E$ be a positive relation on $X$, and $R$ a relation such that $E = R^o R$. Given a regular epimorphism $f: X \twoheadrightarrow Y$ a regular epimorphism,
		%		one has $f(E) = f E f^o = f R^o R f^o = (R f^o)^o Rf^o$.
		
	\end{proof}
\end{lemma}

The following characterisation of Goursat categories through positive relations will be useful in the sequel:

\begin{proposition}\label{Goursat positive}\emph{\cite{GRT}}
	A regular category $\C$ is a Goursat category if and only if any reflexive and positive relation in $\C$ is an equivalence relation.
\end{proposition}

Let us begin with the following observation:

\begin{proposition}	\label{Goursat => TrL}
	In any equivalence distributive Goursat category $\C$, the Trapezoid Lemma holds when $S$ is a reflexive relation and $R$ and $T$ are equivalence relations.
\end{proposition}

\begin{proof} The proof of this result is based on that of Proposition 4.4 in~\cite{GRT} which claims that a Goursat category satisfies the Shifting Lemma when $S$ is a reflexive relation and $R$ and $T$ are equivalence relations.
	
	Let $R$ and $T$ be equivalence relations and let $S$ be a reflexive relation on an object $X$ such that $R\wedge S\leqslant T$. Suppose that we have $(x,y)\in T$, $(x,u) \in S$, $(y,v)\in S$ and $(u,v)\in R$
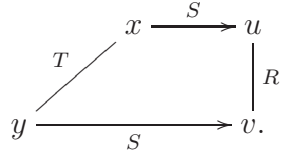
	$$	\xymatrix@C=30pt{
		& x \ar@{-}[dl]_-T \ar[r]^-S & u \ar@{-}[d]^-R \\
		y \ar[rr]_-S  & & v. }$$
	We are going to show that $(u,v)\in T$ .
	
	Consider the two relations $P$ and $W$ on $S$ defined as follows: \\
	$(aSb, cSd)\in P$ if and only if $aRc$ and $bRd$:
	$$
	
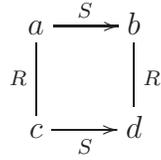
\xymatrix{ 	a \ar[r]^S \ar@{-}[d]_R & b \ar@{-}[d]^R \\
		c \ar[r]_S  & d
	}
	$$
	while $(aSb, cSd)\in W$ if and only if $aTc$ and $bRd$:
	$$
	
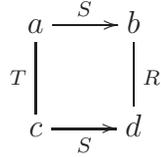
\xymatrix{a \ar[r]^S \ar@{-}[d]_T & b \ar@{-}[d]^R \\
		c \ar[r]_S  & d
	}
	$$
	The relations $P$ and $W$ are  equivalence relations on $S$ since $R$ and $T$ are both equivalence relations.
	Given the equivalence relations $P$, $\Eq(s_2)$ and $W$ on $S$, since $\C$ is Goursat category, one has $$
	\begin{array}{lcl}
	(\,P\wedge \Eq(s_2)\,)\, \vee\, W & = & (\,P\wedge \Eq(s_2)\,)\,W\,(\, P\wedge \Eq(s_2)\,) \vspace{3pt}\\
	& = & W\,(\,P\wedge \Eq(s_2)\,)\,{W,}
	\end{array}
	$$
	which is an equivalence relation (Theorem~\ref{Mal'tsev chars} (iii)).
	
	Since $$P \wedge \Eq(s_2)\leqslant (\,P\wedge \Eq(s_2)\,)\,\vee\, W$$ and $\C$ is a Goursat and equivalence distributive category, by Proposition \ref{Goursat and dist}, we can apply the Trapezoid Lemma to the following diagram
	
	$$
	
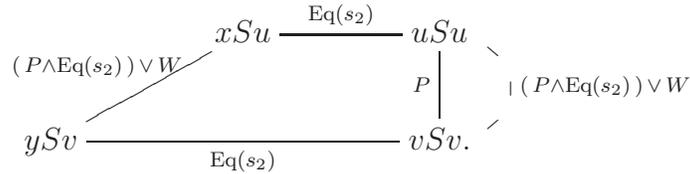
\xymatrix@C=45pt{
		& xSu \ar@{-}[dl]_-{(\, P\wedge \Eq(s_2)\,)\, \vee\,W}  \ar@{-}[r]^-{\Eq(s_2)} & uSu \ar@{-}[d]_-P \ar@(r,r)@{--}[d]^-{(\, P\wedge \Eq(s_2)\,)\, \vee\,W} \\
		ySv   \ar@{-}[rr]_-{\Eq(s_2)}  & & vSv. }
	$$
	
	Note that, $uSu$ and $vSv$ by the reflexivity of $S$. We then obtain $$(uSu, vSv)\in (\, P\wedge \Eq(s_2)\,)\, \vee\,W =(\,P\wedge \Eq(s_2)\,)\,W\,(\, P\wedge \Eq(s_2)\,),$$ this means that there are $a$ and $b$ in $X$ such that
	$$
	(uSu) \left( P \wedge \Eq(s_2) \right) (aSu) W (bSv) \left( P \wedge \Eq(s_2) \right) (vSv),
	$$
	i.e.
	$$
	\xymatrix{
		u  \ar[r]^S  \ar@{-}[d]_R & u \ar@{-}[d]^R \\
		a \ar[r]^S  \ar@{-}[d]_T & u \ar@{-}[d]^R \\
		b \ar@{-}[d]_R \ar[r]^S &  v \ar@{-}[d]^R \\
		v \ar[r]_S & v.
	}
	$$
	Since $aRu$ ($R$ is symmetric), $aSu$ and $R\wedge S\leqslant T$, it follows that $aT u$; similarly one checks that $bTv$. From $uTa$ ($T$ is symmetric), $aTb$ and $bTv$, we
	conclude that $uTv$ ($T$ is transitive), as desired.
\end{proof}	

Since the Trapezoid Lemma implies the Triangular Lemma, we get the following:

\begin{corollary}	\label{Goursat => TL}
	In any equivalence distributive Goursat category $\C$, the Triangular Lemma holds when $S$ is a reflexive relation and $R$ and $T$ are equivalence relations.
\end{corollary}

We are now ready to prove the main result in this section:

\begin{theorem}	\label{Goursat <=> TTrL}
	Let $\C$ be a regular category. The following conditions are equivalent:
	\begin{enumerate}
		\item[(i)] $\C$ is an equivalence distributive Goursat category;
		\item[(ii)] the Trapezoid Lemma holds in $\C$  when $S$ is a reflexive relation and $R$ and $T$ are reflexive and positive relations;
		\item[(iii)] the Triangular Lemma holds in $\C$  when $S$ is a reflexive relation and $R$ and $T$ are reflexive and positive relations.
	\end{enumerate}
\end{theorem}	

\begin{proof}
	(i) $\Rightarrow$ (ii)	Since $\C$ is a Goursat category, by Proposition \ref{Goursat positive}, reflexive and positive relations are necessarily equivalence relations. Since $\C$ is also equivalence distributive, by Proposition \ref{Goursat => TrL}, the Trapezoid Lemma holds when $S$ is a reflexive relation and $R$ and $T$ are reflexive and positive relations.\\
	(ii) $\Rightarrow$ (iii) is obvious.\\
	(iii) $\Rightarrow$ (i) We follow the proof of Theorem 4.6 of~\cite{GRT} with respect to the implication: if the Shifting Lemma holds in $\C$ when $S$ is a reflexive relation and $R$ and $T$ are reflexive and positive relations, then $\C$ is a Goursat category. The main issue is to fit the rectangle to which we applied the Shifting Lemma in that result, into a suitable triangle to which we shall now apply the Triangular Lemma (to get the same conclusion that $\C$ is a Goursat category).
	
	To prove that $\C$ is a Goursat category, we show that for any reflexive relation $E$ on $X$ in $\C$, $EE^\circ= E^\circ E$ (Theorem~\ref{Goursat chars}(v)). Suppose that $(x,y)\in EE^\circ$. Then, for some $z$ in $X$, one has that $(z,x)\in E$ and $(z,y)\in E$. Consider the reflexive and positive relations $EE^\circ$ and $E^\circ E$, and the reflexive relation $E$ on $X$. From the reflexivity of $E$, we get $E\leqslant EE^\circ$ and $E\leqslant E^\circ E$; thus $EE^\circ \wedge E=E\leqslant E^\circ E$. We may apply our assumption (for $R=EE^{\circ}, S=E, T=E^{\circ}E$) to the following relations given in solid lines:
	
	$$
	
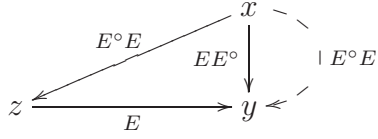
\xymatrix@C=75pt{
		{} & x \ar[d]_-{EE^\circ} \ar@(r,r)@{-->}[d]^-{E^\circ E} \ar[dl]_{E^\circ E}   \\
		z   \ar[r]_-E  & y }
	$$
	to conclude that $(x,y)\in E^\circ E$. Having proved that $EE^\circ \leqslant E^\circ E$ for \emph{every} reflexive relation $E$, the equality $E^\circ E\leqslant EE^\circ$ follows immediately.
\end{proof}

We finish this section with a characterisation of equivalence distributive Goursat categories through a property on ternary relations which is stronger than condition (iv) of Theorem~\ref{Goursat chars}. The process to obtain such a characterisation is similar to what was done to obtain Theorem~\ref{ternary difunctionality} for the Mal'tsev context. Therefore, we only give the main features leaving the proof for the reader to complete.

Recall from~\cite{Lipparini} that a $3$-permutable congruence distributive variety is such that there exists ternary terms $r(x,y,z)$ and $s(x,y,z)$ such that $r(x,y,y)=x, r(x,x,y)=s(x,y,y), s(x,x,y)=y$ and $r(x,y,x)=x=s(x,y,x)$. It is easy to check that, equivalently, such varieties admit quaternary terms $p(x,y,z,w)$ and $q(x,y,z,w)$ such that $p(x,y,y,z)=x$, $p(x,x,y,y)=q(x,x,y,y)$, $q(x,y,y,z)=z$ and $p(x,y,z,x)=x=q(x,y,z,x)$.

These Mal'tsev conditions translate into a property on relations (following the technique in~\cite{JanelidzeV}) which is expressed for ternary relations $P\rightarrowtail (X\times A)\times U$, seen as binary relations from $X\times A$ to $U$, as

\begin{equation}\label{PP}
\begin{array}{c}
(x,a) P u \\
(y,b) P u \\
(y,c) P v \\
(z,a) P v \\
\hline
(x,a) P w \\
(z,a) P w,
\end{array}
\end{equation}
for some $w$ in $U$.

In a regular context, property \eqref{PP} is equal to
$$
\Eq(\pi_A)\wedge PP^{\circ}\Eq(\pi_X)PP^{\circ}\leqslant PP^{\circ},
$$
and one can prove the following:
\begin{theorem}\label{quaternary difunctionality}
	Let $\C$ be a regular category. The following conditions are equivalent:
	\begin{enumerate}
		\item[(i)] $\C$ is an equivalence distributive  Goursat category;
		\item[(ii)] any relation $P\rightarrowtail (X\times A)\times U$ has property \eqref{PP}.
	\end{enumerate}
\end{theorem}


\begin{thebibliography}{10}
	\bibitem{Barr} M. Barr, P. A. Grillet and D. H. van Osdol, \emph{Exact categories and categories of sheaves}, Lecture Notes in Math. 236, Springer-Verlag (1971).
	\bibitem{BB} F. Borceux, D. Bourn, \emph{Mal'cev, protomodular, homological and semi-abelian categories}, Kluwer, (2004).
	\bibitem{DBourn} D. Bourn, \emph{Congruence distributivity in Goursat and Mal'tsev categories}, Appl. Categ. Structures (2005) 13: 101-111.
	\bibitem{BG0} D. Bourn, M. Gran, \emph{Categorical Aspects of Modularity}, Galois Theory, Hopf Algebras and Semiabelian Categories, Fields Instit. Commun., 43, Amer. Math. Soc., Providence RI (2004) 77-100.
	\bibitem{BG1} D. Bourn, M. Gran, \emph{Normal sections and direct product decompositions}, Comm. Algebra. 32, no. 10 (2004) 3825-3842.
	\bibitem{BGJ} D. Bourn, M. Gran, P.-A. Jacqmin, \emph{On the naturalness of Mal'tsev categories}, preprint arXiv:1904.06719 (2019).
	\bibitem{CKP} A. Carboni, G.M. Kelly, M.C. Pedicchio, \emph{Some remarks on Maltsev and Goursat categories}, Appl. Categ. Structures 1, no.4 (1993) 385-421.
	\bibitem{CLP} A. Carboni, J. Lambek, M.C. Pedicchio, \emph{Diagram chasing in Mal'cev categories}, J. Pure Appl. Algebra 69, no. 3 (1991) 271-284.
	\bibitem{CPP} A. Carboni, M.C. Pedicchio, N. Pirovano, \emph{Internal graphs and internal groupoids in Mal'cev categories}, Category theory 1991 (Montreal, PQ, 1991), 97-109, CMS Conf. Proc., 13, Amer. Math. Soc., Providence, RI, 1992.
	\bibitem{Trapezoid} I. Chajda, G. Cz\'edli, E. Horv\'ath, \emph{Trapezoid Lemma and congruence distributivity}, Math. Slovaca 53 (2003), No. 3, 247-253.
	\bibitem{Duda1} J. Duda, \emph{The Upright Principle for congruence distributive varieties}, Abstract of a
	seminar lecture presented in Brno, March, 2000.
	\bibitem{Duda2} J. Duda, \emph{The Triangular Principle for congruence distributive varieties}, Abstract of a
	seminar lecture presented in Brno, March, 2000.
	\bibitem{GRT} M. Gran, D. Rodelo, I. Tchoffo Nguefeu, \emph{Variations of the Shifting Lemma and Goursat categories}, Algebra Universalis (2018) 80:2.
	\bibitem{GR} M. Gran, J. Rosick\'y, \emph{Special reflexive graphs in modular varieties}, Algebra Universalis 52 (2004) 89-102.
	\bibitem{GSV} M. Gran, F. Sterck, J. Vercruysse, \emph{A semi-abelian extension of a theorem by Takeuchi}, J. Pure Appl. Algebra 223 no. 10 (2019) 4171-4190.
	\bibitem{Gumm} H.P. Gumm, \emph{Geometrical methods in congruence modular algebras}, Mem. Amer. Math. Soc. 45 (1983) no. 286.
	\bibitem{HM} J. Hagemann, A. Mitschke, \emph{On $n$-permutable congruences}, Algebra Universalis 3 (1973) 8-12.
	\bibitem{Hoefnagel} M. Hoefnagel, \emph{Majority categories}, Theory Appl. Categories 34 (2019)  249-268.
	\bibitem{Hoef2} M. Hoefnagel, \emph{Characterizations of majority categories}, Appl. Categ. Structures 28 (2020) 113-134.
	\bibitem{Hoef3} M. Hoefnagel, \emph{A categorical approach to lattice-like structures}, Ph.D thesis (2018)
	\bibitem{GJ} G. Janelidze, \emph{A history of selected topics in categorical algebra I: From Galois theory to abstract commutators and internal groupoids}, Categories and General Algebraic Structures with Applications 5, no. 1 (2016) 1-54.
	\bibitem{JanelidzeI} Z. Janelidze, \emph{Closedness properties of internal relations I: a unified approach to Mal'tsev, unital and subtractive categories}, Th. Appl. Categ, Vol. 16, No. 12 (2006) 236–261.
    \bibitem{JanelidzeV} Z. Janelidze, \emph{Closedness properties of internal relations V: linear Mal’tsev conditions}, Algebra Universalis 58 (2008) 105-117.
	\bibitem{JRVdL} Z. Janelidze, D. Rodelo, T. Van der Linden,  \emph{Hagemann's theorem for regular categories}, J. Homotopy Relat. Structures 9(1) (2014) 55-66.
	\bibitem{Jonsson53} B. J\'{on}sson, \emph{On the representation of lattices}, Math. Scand. 1 (1953) 193--06.
	\bibitem{Lipparini} P. Lipparini, \emph{n-Permutable varieties satisfy non trivial congruence identities}, Algebra Universalis \textbf{33} (2018) 159-168.
	\bibitem{Malcev} A.I. Mal'cev, \emph{On the general theory of algebraic systems}, Mat. Sbornik N.S. 35 (1954), 3-20.
	\bibitem{Mit} A. Mitschke, \emph{Implication algebras are $3$-permutable and $3$-distributive} Algebra Universalis \textbf{1} (1971) 182-186.
	\bibitem{Pedic} M. C. Pedicchio, \emph{Arithmetical categories and commutator theory}, Appl. Categ. Structures 4, (1996) 297-305.
	\bibitem{Pixley} A.F. Pixley, \emph{Characterizations of arithmetical varieties}, Algebra Universalis \textbf{9} (1979) 87-98.
	\bibitem{Riguet} J. Riguet, \emph{Relations binaires, fermeture, correspondances de Galois}, Bulletin de la Soci\'et\'e Math\'ematique de France \textbf{76} (1948) 114-155.
	\bibitem{Sel} P. Selinger, \emph{Dagger Compact Closed Categories and Completely Positive Maps}, Electronic Notes in Theoretical Computer Science 170 (2007) 139-163.
	\bibitem{Smith} J.D.H. Smith, \emph{Mal'cev varieties}, Lecture Notes in Math. 554, Springer-Verlag (1976).
\end{thebibliography}
\end{document}